\documentclass[12pt]{amsart}

\usepackage{mathtools, titlesec, tikz-cd, amsmath, amsthm, amssymb, enumerate, enumitem, mathrsfs, verbatim, hyperref, fullpage, wasysym, cleveref}

\titleformat{name=\section}{}{\thetitle.}{0.8em}{\centering\scshape}
\titleformat{name=\subsection}[runin]{}{\thetitle.}{0.8em}{\bfseries}[.]
\titleformat{name=\subsubsection}[runin]{}{\thetitle.}{0.8em}{\bfseries}[.]

\theoremstyle{plain}
\newtheorem{thm}{Theorem}[section]
\newtheorem{theorem}[thm]{Theorem}
\newtheorem*{theorem*}{Theorem}
\newtheorem{lemma}[thm]{Lemma}
\newtheorem*{lemma*}{Lemma}
\newtheorem{corollary}[thm]{Corollary}
\newtheorem{proposition}[thm]{Proposition}
\newtheorem*{proposition*}{Proposition}

\newtheorem*{conjecture*}{Conjecture}

\newtheorem*{question*}{Question}

\newtheorem*{fact*}{Fact}

\theoremstyle{definition}
\newtheorem{definition}[thm]{Definition}

\newtheorem*{conventions*}{Conventions}

\newtheorem*{acknowledgements*}{Acknowledgements}
\newtheorem{example}[thm]{Example}
\newtheorem{assumption}[thm]{Assumption}
\newtheorem{construction}[thm]{Construction}
\newtheorem{remark}[thm]{Remark}

\DeclareMathOperator{\co}{co}

\DeclareMathOperator{\Frac}{Frac}
\DeclareMathOperator{\trdeg}{trdeg}
\DeclareMathOperator{\can}{can}
\DeclareMathOperator{\id}{id}

\DeclareMathOperator{\Spec}{Spec}
\DeclareMathOperator{\spl}{sp}

\newcommand{\bphi}{\boldsymbol{\Phi}}

\DeclareMathOperator{\Gal}{Gal}
\newcommand{\Dgal}{\Gal^{\delta}}

\DeclareMathOperator{\Ind}{Ind}

\DeclareMathOperator{\M}{M}

\DeclareMathOperator{\PGL}{PGL}

\DeclareMathOperator{\GL}{GL}

\DeclareMathOperator{\Mor}{Mor}
\DeclareMathOperator{\Hom}{Hom}

\DeclareMathOperator{\Alg}{Algs}

\DeclareMathOperator{\Aut}{Aut}

\newcommand{\Daut}{\Aut^\delta}
\DeclareMathOperator{\UDaut}{\underline{Aut}^\delta}

\DeclareMathOperator{\PV}{PV}
\DeclareMathOperator{\Object}{-Object}
\newcommand{\difftors}[1]{\delta\operatorname{\mathbf{-}}#1\operatorname{\mathbf{-torsor}}}
\DeclareMathOperator{\diffmod}{\delta\mathbf{-mod}}
\DeclareMathOperator{\DCSA}{\delta\mathbf{-CSA}}
\DeclareMathOperator{\TF}{TF}
\DeclareMathOperator{\tors}{\mathbf{-torsor}}

\title{Cohomology for Picard-Vessiot Theory}
\author{Man Cheung Tsui and Yidi Wang}

\begin{document}
\maketitle

\begin{abstract}
We introduce a cohomology theory that classifies differential objects that arise from Picard-Vessiot theory, using differential Hopf-Galois descent. To do this, we provide an explicit description of Picard-Vessiot theory in terms of differential torsors. We then use this cohomology to give a bijective correspondence between differential objects and differential torsors. As an application, we prove a universal bound for the differential splitting degree of differential central simple algebras.
\end{abstract}

\section*{Introduction}
\label{section:introduction}
\addcontentsline{toc}{section}{Introduction}

Picard-Vessiot theory studies algebraic relations between solutions of homogeneous linear differential equations via Picard-Vessiot rings, which are analogues to Galois extensions from algebra. 
While various aspects of algebra have been effectively studied under the framework of torsors and Galois cohomology, Picard-Vessiot theory has primarily been studied by treating the algebraic and differential aspects rather separately.
The next paragraphs discuss the two things this paper does to bring the differential theory closer in line with the algebraic theory. 

For the rest of this paper, ``$\delta$'' will serve both as our shorthand for the word ``differential'' and our notation for derivations. We will also fix a $\delta$-field $(F,\delta)$, i.e., a field $F$ equipped with a derivation $\delta:F\to F$. We will assume that $F$ is of characteristic zero, and that its field of constants $C$ is algebraically closed. For instance, the reader can take $(F,\delta)$ to be the complex function field $\mathbb C(x)$ with the derivation $d/dx$.

The first objective of this paper is to better describe Picard-Vessiot theory using $\delta$-torsors. In general, a Picard-Vessiot ring of an $n\times n$ matrix differential equation $y'=Ay$ over $F$ is constructed as follows.
\begin{enumerate}
    \item\label{PV-construction-1} Define the $\delta$-ring $R\coloneqq F[X_{ij}, \det(X_{ij})^{-1}\mid 1\le i,j\le n]$ with derivation $\delta(X_{ij}) = A(X_{ij})$.
    \item\label{PV-construction-2} Let $\mathfrak{m}$ be a maximal $\delta$-ideal of $R$. Then $R/\mathfrak{m}$ is a Picard-Vessiot ring for $y'=Ay$.
\end{enumerate}
In the classical theory, Kolchin had shown that the spectrum of a Picard-Vessiot ring is a torsor for its $\delta$-Galois group (see e.g., \cite{kolchin1973differential} or \cite[Theorem 1.29]{vdPS03}). In \cite{BHHW18}, it was further shown that such a torsor is a $\delta$-torsor.
In \Cref{section:differential_torsors}, we give refinements of statements in \cite{BHHW18}. Namely, we strengthen Step \ref{PV-construction-1} into a bijective correspondence between $\delta$-isomorphism class of $\Spec(R)$ and the original equation $y' = Ay$ in \Cref{prop:cohomology_gln}. 
We then generalize Step \ref{PV-construction-2} to arbitrary $\delta$-torsors in \Cref{thm:comprehensive_induced_torsor}.
As a consequence, we show that any $\delta$-torsor arises from a unique Picard-Vessiot ring up to isomorphism.

The second objective of this paper is to construct a cohomology theory for Picard-Vessiot theory that plays a role similar to Galois cohomology in Galois theory.
We will use this cohomology to classify various differential objects that arise in Picard-Vessiot theory.
Namely, we give a bijective correspondence between certain objects ($\bphi$-objects) and $\delta$-torsors of linear algebraic groups. 
To establish this correspondence, we introduce a cohomology theory using $\delta$-Hopf-Galois descent. The $1$-cocycles in this cohomology $H^{1}_{\delta}(\Gamma, G)$ are morphisms $\Gamma\to G$ of varieties over $C$, where $\Gamma$ and $G$ are linear algebraic groups over $C$.

Fix a Picard-Vessiot extension $K/F$ and a $\bphi$-object $M$ over $F$ (e.g., $\delta$-module, $\delta$-torsor, $\delta$-central simple algebra). Here a $\bphi$-object refers to an object of type $\bphi$; see \Cref{subsection:phi-structure} for the exact definition. We will classify the set $\TF(K/F,M)$ of twisted forms of a $\bphi$-object $M$, i.e., $\bphi$-objects $N$ over $F$ that become isomorphic to $M$ over $K$ by the above cohomology via the following theorem. 

\begin{theorem*}[\Cref{thm:tf-over-fields}]
     Let $K/F$ be a Picard-Vessiot extension with $\delta$-Galois group $\Gamma$. Let $M$ be a $\bphi$-object over $F$. Suppose that the group functor
     \[
        \underline{\operatorname{Aut}}^{\delta, \bphi}(M)\colon (C\operatorname{-Algebras}) \to (\operatorname{Groups})\colon
        D \mapsto \operatorname{Aut}^{\delta, \bphi}_{F \otimes_C D}(M\otimes_C D)
     \]
     is represented by a linear algebraic group $G$ over $C$. Here the superscript $\bphi$ indicates that the automorphisms preserve the type $\bphi$. Then the two sets $\TF(K/F,M)$ and $H^{1}_{\delta}(\Gamma,G)$ are in bijection. 
\end{theorem*}
A similar correspondence was also considered by \cite{masuoka2021twisted} where our notion of twisted forms was classified in Amitsur cohomology, using faithfully flat descent.  One advantage of our cohomology theory is that it allows us to work directly with Picard-Vessiot rings, which are generally not faithfully flat. Furthermore, our cohomology does not require information about the Picard-Vessiot extension like Amitsur cohomology does and thereby enables us to directly work with the $\delta$-Galois groups. Recently, \cite{MP} also proved \Cref{thm:tf-over-fields} when the twisted forms are specifically $\delta$-torsors in a model-theoretic approach.

When $G$ is $\PGL_{n}$, \Cref{thm:tf-over-fields} establishes the bijective correspondence between $\delta$-central simple algebras of degree $n$ introduced in \cite{Juan_Magid08}, and $\delta$-$\PGL_{n,F}$-torsors; see \Cref{differential_pgln}. Consequently, we can address questions raised for $\delta$-central simple algebras, leveraging properties of $\delta$-$\PGL_{n,F}$-torsors. Previously, \cite{gupta2022splitting} proved that the $\delta$-splitting degree for $\delta$-quaternion algebras is at most three. Their approach involved explicitly solving the Riccati equations associated to $\delta$-quaternion algebras. However, the algebraic structure of general $\delta$-central simple algebras is much less explicit and therefore one cannot expect a nice associated differential equations for computation. Instead, we will use our new bijective correspondence with $\delta$-$\PGL_{n,F}$-torsors to prove the following universal bound on the $\delta$-splitting degree for general $\delta$-central simple algebras:

\begin{theorem*}[\Cref{cor:bound_DCSA}]
    Let $(A, \delta)$ be a $\delta$-central simple $F$-algebra of degree $n$. Then
    \begin{equation*}
        \deg_{\spl}^{\delta} (A, \delta) \leq \dim(\PGL_n) = n^2 - 1.
    \end{equation*} 
    This bound is optimal if $\PGL_n$ appears as a $\delta$-Galois group over the $\delta$-field $F$.
\end{theorem*}

\bigskip

Here is how this manuscript is organized. \Cref{section:pv} reviews Picard-Vessiot theory. 
\Cref{section:differential_torsors} studies $\delta$-torsors as introduced in \cite{BHHW18}. 
\Cref{section:twisted-forms-descent} defines the twisted forms that we will use. 
\Cref{section:cohom} proves \Cref{thm:tf-over-fields} as part of the more general correspondence (\Cref{thm:F-G-inverse})
about $\delta$-Hopf-Galois extensions.
Then \Cref{prop:absolute-cohom} restates \Cref{section:cohom} in terms of absolute cohomology.
In \Cref{section:DCSA} we recall $\delta$-central simple algebras as introduced in \cite{Juan_Magid08} and answer a question of \cite{gupta2022splitting}.
\bigskip

\begin{acknowledgements*}
    Sections \ref{section:twisted-forms-descent} and \ref{section:cohom} contain results from the first author's thesis. The second author was partially supported by NSF grant DMS-2102987. Both authors thank David Harbater and Mark van Hoeij for helpful conversations, and Julia Hartmann for her patient advising and ample support.
\end{acknowledgements*}

\section{Picard-Vessiot theory}
\label{section:pv}

We first review Picard-Vessiot theory mainly following \cite[Chapter 1]{vdPS03}.

A \emph{$\delta$-ring} is a commutative ring $R$ equipped with a derivation $\delta\colon R\to R$, i.e., an additive map $\delta:R\to R$ satisfying $\delta(rs)=r\delta(r)s+r\delta(s)$ for all $r,s\in R$. 
Let $(R, \delta)$ be a $\delta$-ring.
The \emph{ring of constants} of $R$ is the subring $R^{\delta} \coloneqq \{r\in R\mid \delta(r) = 0\}$
of $R$. 
If $R^{\delta}$ is a field, we call $R^{\delta}$ the \emph{field of constants} of $R$. 
An ideal $I$ of $R$ is a \emph{$\delta$-ideal} if $\delta(I) \subseteq I$. If the only $\delta$-ideals of $R$ are $0$ and $R$, then $R$ is \emph{simple}.
A \emph{homomorphism} of $\delta$-rings $(R,\delta_R)$ and $(S,\delta_S)$ is a ring homomorphism $\phi:R\to S$ such that $\phi\circ \delta_R = \delta_S\circ \phi$. A $\delta$-ring $S$ is a \emph{simple $\delta$-quotient} of $R$ if $S$ is simple and $S$ is the image of a $\delta$-ring homomorphism $R\twoheadrightarrow S$.

A \emph{$\delta$-$R$-module} is an $R$-module $M$ equipped with a derivation $\delta:M\to M$, i.e., $\delta(rm) = r\delta(m)+\delta(r)\delta(m)$ for all $r\in R$, $m\in M$. 
Let $M$ be a $\delta$-$F$-module with $\dim_F M = n<\infty$. 
After a choice of $F$-basis, $M$ corresponds to an $n\times n$ matrix differential equation $y' = Ay$ with $A\in \M_n(F)$.
Two $\delta$-$F$-modules are isomorphic if and only if their corresponding equations $y'=Ay$ and $z'=Bz$ are \emph{gauge equivalent}, i.e., $B=P'P^{-1}+PAP^{-1}$ for some $P\in\GL_n(F)$. We call $M$ \emph{trivial} if $M\cong F\otimes_C M^{\delta}$ as $\delta$-$F$-modules, where $M^{\delta}\coloneqq \{m\in M\mid \delta(m) = 0\}$. In a suitable $F$-basis, a trivial $\delta$-module  corresponds to the equation $y'=0\cdot y$.

Let $(R,\delta_R)$ be a $\delta$-ring.  A $\delta$-$R$-algebra is a $\delta$-ring $(S,\delta_S)$ such that $S$ is an $R$-algebra and $\delta_S$ extends $\delta_R$. Concepts like \emph{$\delta$-$R$-coalgebras}, \emph{$\delta$-Hopf algebras}, and \emph{$\delta$-field extensions} are similarly defined as $R$-coalgebras, Hopf algebras, and field extensions with compatible derivations that commute with the structure ring homomorphisms.

A \emph{Picard-Vessiot ring} for a differential equation $y' = Ay$ over $F$ is a simple $\delta$-$F$-algebra satisfying the following conditions:    
\begin{enumerate}
    \item There exists a fundamental matrix $Y \in \GL_n(R)$ that satisfies $Y' = AY$;
    \item $R = F\left[Y_{ij}, \det(Y_{ij})^{-1}\mid 1\leq i, j\leq n\right]$;
    \item $R^{\delta} = C$.
\end{enumerate}
A \emph{Picard-Vessiot extension} $L$ for $y' = Ay$ is the fraction field of a Picard-Vessiot ring.  The field of constants of a Picard-Vessiot field is known to be $C$. Since we assume $C$ to be algebraically closed, a Picard-Vessiot ring for $y'=Ay$ over $F$ exists and is unique up to an isomorphism of $\delta$-$F$-algebras (\cite[Proposition 1.20]{vdPS03}).

Given a Picard-Vessiot ring $R/F$, the group functor
\[
\UDaut(R/F)\colon (C\operatorname{-Algebras}) \to (\operatorname{Groups})\colon 
D \mapsto \Daut_{F \otimes_C D}(R\otimes_C D)
\]
is represented by a linear algebraic group over $C$. Replacing $R$ by $\Frac(R)$ gives a group functor $\UDaut(\Frac(R)/F)$ represented by the same linear algebraic group. See \cite[Corollary 2.12]{dyckerhoff2008inverse}. The \emph{$\delta$-Galois group} for $R/F$, denoted by $\Dgal(R/F)$, is the linear algebraic group over $C$ representing $\UDaut(R/F)$.

Finally the \emph{Picard-Vessiot closure} $F^{\PV}$ of $F$ is the direct limit of Picard-Vessiot fields $K$ over $F$, filtered by inclusion. It exists and is unique up to an isomorphism of $\delta$-$F$-algebras; see \cite[Section 3]{magid2002picard}.

\section{Differential torsors}
\label{section:differential_torsors}

\subsection{Differential torsors}
\label{subsection:differential_torsors}
Following \cite{BHHW18}, we reformulate Picard-Vessiot extensions as $\delta$-torsors.

\begin{definition}
    Let $G$ be a linear algebraic group over $C$. Let $R$ be a $\delta$-ring containing $C$. A \emph{$G_R$-torsor} is an affine $R$-variety $X$ equipped with a simply transitive group action $\alpha \colon X \times G_R \to X$. A \emph{$\delta$-$G_R$-torsor} is a $G_F$-torsor $X$ equipped with a derivation $\delta\colon R[X]\to R[X]$ extending the derivation $\delta$ on $R$, such that $\rho \colon R[X] \otimes_R R[X] \to R[X] \otimes_R R[G_F]$ corresponding to the morphism $(\alpha, \operatorname{pr}_1) \colon X \times G_R \to X \times X$ is a $\delta$-$R$-algebra homomorphism. A $\delta$-torsor is \emph{simple} if $R[X]$ is a simple $\delta$-ring.
\end{definition}

Let $R/F$ be a Picard-Vessiot ring with $\delta$-Galois group $G$. 
Then $\Spec(R)$ is a $G_F$-torsor by the Torsor Theorem (\cite[Theorem 1.28]{vdPS03}) and a simple $\delta$-$G_F$-torsor by \cite[Proposition 1.12]{BHHW18}.

\begin{definition}
    Let $X$ be a $\delta$-$G_F$-torsor with derivation $\delta$ on $F[X]$. We call $X$ a \emph{trivial} $\delta$-torsor if $X$ is a trivial $G_F$-torsor and $F[X] = F \otimes_C F[X]^{\delta}$, where $F[X]^{\delta}$ denotes the constants of $F[X]$ under the derivation $\delta$.
\end{definition}

\begin{example}
\label{ex:diff_gln}
    By Hilbert's Theorem 90, any $\GL_{n,F}$-torsor $X$ is trivial, so $X$ can be identified with $\GL_{n,F}$ and its coordinate ring with
    \[
        F[X] = F[\GL_{n,F}] = F\left[X_{ij}, \det(X_{ij})^{-1}\mid 1\leq i, j\leq n\right].
    \]
    \cite[Examples 1.9 and 1.10]{BHHW18} further show that all the derivations that turn $X$ into a $\delta$-$\GL_{n,F}$-torsor are given by $\delta_A(X_{ij}) = A(X_{ij})$ for any $A \in \M_n(F)$.
    In particular, a trivial torsor can be a nontrivial $\delta$-torsor.
\end{example}

We now show that the above description of $\delta$-$\GL_n$-torsors respects isomorphisms. Let $\diffmod_{n}(F)$ be the set of isomorphism classes of $\delta$-$F$-modules of dimension $n$. We also let $\difftors{G_F}(F)$ to be the set of isomorphism classes of $\delta$-$G_{F}$-torsors.

\begin{proposition}\label{prop:cohomology_gln}
    The map 
    \[
    \mathcal{F}\colon \diffmod_n(F) \to \difftors{\GL_{n,F}}(F)
    \]
    taking the $\delta$-$F$-module for $y' = Ay$ to the $\delta$-$\GL_{n, F}$-torsor $Y$ with coordinate ring 
    \[
    F[Y] = F[Y_{ij}, \det(Y_{ij})^{-1}]\quad\text{and}\quad (Y_{ij})' = A(Y_{ij})
    \]
    is both functorial in $F$ and a bijection. In particular, when $A=0$, the map $\mathcal{F}$ takes the trivial $\delta$-$F$-module to the trivial $\delta$-$\GL_{n,F}$-torsor.
\end{proposition}

\begin{proof}
    Functoriality is clear; surjectivity follows from \Cref{ex:diff_gln}. The well-definedness and injectivity of $\mathcal{F}$ follow from the following claim:\bigskip

    \emph{Let $A,B\in \M_n(F)$. Two differential equations $y'=Ay$ and $z'=Bz$ are gauge equivalent if and only if their corresponding $\delta$-$\GL_{n,F}$-torsors are isomorphic.}\bigskip

    To prove the claim, let $Y$ and $Z$ be $\delta$-$\GL_{n,F}$-torsors corresponding to $y'=Ay$ and $z'=Bz$. We identify $Y$ and $Z$ with $X=\GL_{n,F}$ as $\GL_{n,F}$-torsors. We also have derivations
    \[
    \delta_{A}(X_{ij}) = A(X_{ij})\text{ on }F[Y]=F[X]\qquad\text{and}\qquad \delta_{B}(X_{ij}) = B(X_{ij})\text{ on }F[Z]=F[X].
    \]
    Since an automorphism of $\GL_{n,F}$ as a (right) $\GL_{n,F}$-torsor is given by $\varphi:x\mapsto Px$ for some $P\in \GL_n(F)$, this $\varphi$ defines a $\delta$-$\GL_{n,F}$-torsor isomorphism between $Y$ and $Z$ if and only if 
    $\varphi^{\ast}\circ \delta_{B} = \delta_{A}\circ \varphi^{\ast}$. 
    But
    \[
    \delta_{A}\circ \varphi^{\ast} (X_{ij}) = \delta_A(P(X_{ij})) = P'(X_{ij}) + PA(X_{ij})
    \]
    and
    \[
    \varphi^{\ast}\circ \delta_{B} (X_{ij}) = \varphi^{\ast}(B(X_{ij})) = B \varphi^\ast (X_{ij}) = BP (X_{ij})
    \]
    are equal precisely when $B = P' P^{-1}+PAP^{-1}$, so the claim follows.
\end{proof}

\subsection{Induced differential torsors}
\label{subsection:induced_differential_torsors}

\begin{definition}
    Let $\psi \colon H \hookrightarrow G$ be an embedding of linear algebraic groups over $C$. Let $Y$ be a $\delta$-$H_F$-torsor. The \emph{induced $\delta$-$G_F$-torsor via $\psi$}, denoted $\Ind^{\psi}(Y)$, is defined as the spectrum of the ring $(F[Y] \otimes_C C[G])^H$ equipped with the derivation restricted from $F[Y] \otimes_C C[G]$ by viewing elements in $C[G]$ as constants. 
\end{definition}

The construct $\Ind^\psi(Y)$ is known to be a $\delta$-$G_{F}$-torsor. 
Moreover, the choice of $\psi$ uniquely determines the derivation on the coordinate ring of the induced $\delta$-torsor (see \cite[Proposition 1.8]{BHHW18}).

\begin{remark}\label{base_change_induced}
    Induction of $\delta$-torsors commutes with base change. To see this, let $L/F$ be a $\delta$-field extension for which $L^{\delta} = F^{\delta}$. Then \cite[Remark A.9(d)]{BHHW18} gives a canonical isomorphism $\Ind^{\psi}(X \times_F L) \cong \Ind^{\psi}(X) \times_F L$ as $G_L$-torsors. Finally the isomorphism $L[X] \otimes_C C[G] \cong F[X] \otimes_C C[G] \otimes_F L$ of $\delta$-rings restricts to $(L[X] \otimes_C C[G])^{H_L} \cong (F[X] \otimes_C C[G])^H \otimes_F L$, making $\Ind^{\psi}(X \times_F L) \cong \Ind^{\psi}(X) \times_F L$ an isomorphism of $\delta$-torsors.
\end{remark}

Induction of $\delta$-torsors also satisfies the following universal property.

\begin{proposition}\label{prop: universal}
    Let $\psi \colon H \hookrightarrow G$ be an embedding of linear algebraic groups over $C$. Let $Y$ be a $\delta$-$H_F$-torsor and let $Z$ be a $\delta$-$G_F$-torsor. Then there exists a unique derivation on $\Ind^{\psi}(Y)$ such that for every $\psi$-equivariant morphism of $\delta$-torsors $Y \to Z$, there exists a unique isomorphism $\varphi \colon \Ind^\psi(Y) \to Z$ of $\delta$-$G_F$-torsors making the following diagram commute.
    \[
        \begin{tikzcd}[sep=1.4em]
            Y \arrow{rr}{} \arrow{dr}{\alpha} &  & \Ind^{\psi}(Y) \arrow[dl, dashed, "\varphi"]\\
            & Z &
        \end{tikzcd}
    \]
\end{proposition}

\begin{proof}
    By \cite[Proposition A.8]{BHHW18}, $\varphi$ exists and is unique for $G_F$-torsors. 
    It remains to check that $\varphi$ is also an isomorphism of $\delta$-$G_{F}$-torsors. On coordinate rings, the $G_F$-equivariant map $\varphi^{\ast}$ is the restriction of the ring homomorphism
    \begin{center}
        \begin{tikzcd}
            &F[Z] \arrow{r}{\Delta_{F[Z]}} &F[Z] \otimes_C C[G] \arrow{r}{\alpha^{\ast} \otimes 1 } &F[Y] \otimes_C C[G]
        \end{tikzcd}
    \end{center} to the codomain $(F[Y]\otimes_C C[G])^H$. Since $\Delta_{F[Z]}$ and $\alpha^{\ast}$ are $\delta$-ring homomorphisms, $\varphi^{\ast}$ is too, so $\varphi$ is an isomorphism of $\delta$-$G_{F}$-torsors.
\end{proof}

\begin{theorem}\label{thm:comprehensive_induced_torsor}
    Let $X$ be a $\delta$-$G_F$-torsor and let $\phi:G\to \GL_{n}$ be an embedding of linear algebraic groups over $C$. Then $Z\coloneqq \Ind^\phi (X)$ is a $\delta$-$\GL_{n,F}$-torsor with derivation $Z' = AZ$ for some $A\in \M_n(F)$ (by \Cref{prop:cohomology_gln}). Let $H$ be the $\delta$-Galois group of $Z'=AZ$.
    \begin{enumerate}[label=(\alph*)]
        \item\label{induced_a} Any simple $\delta$-quotient $R$ of $F[X]$ is a Picard-Vessiot ring for $Z'=AZ$. In particular, $Y\coloneqq \Spec(R)$ is a simple $\delta$-$H_F$-torsor.
        \item\label{induced_b} There is a closed embedding $\psi \colon H \hookrightarrow G$ and $X \cong \Ind^{\psi}(Y)$.
        \item\label{induced_c} Induction is unique: if $X \cong \Ind^{\psi'}(Y')$ for some $\psi':H'\to G$ and $\delta$-$H'_F$-torsor $Y'$, then $H\cong H'$, and $Y\cong Y'$ as $\delta$-$H_F$-torsors.
        \item\label{induced_d} $Z \cong \Ind^{\rho}(Y)$, where $\rho\coloneqq \phi\circ\psi \colon H \to \GL_n$.
    \end{enumerate}
\end{theorem}

\begin{proof}
    The proof of \cite[Proposition 1.15]{BHHW18} shows the following:\bigskip

    \emph{$(\ast)$ Given a $\delta$-$G_F$-torsor $X$, any simple $\delta$-quotient $R$ of $F[X]$ determines a simple $\delta$-$H$-torsor $Y\coloneqq \Spec (R)$. Moreover, 
    $X \cong \Ind^{\psi}(Y)$
    for some embedding $\psi \colon H\hookrightarrow G$.}\bigskip
    
    Let $R$ be a simple $\delta$-quotient of $F[X]$. Precomposing with $F[Z] \twoheadrightarrow F[X]$ gives $F[Z] \twoheadrightarrow F[X] \twoheadrightarrow R$, making $R$ a simple $\delta$-quotient of $F[Z]$. Therefore, $R$ is a Picard-Vessiot ring of $Z' = AZ$, which is unique up to isomorphism. Now by applying $(\ast)$, $Y \coloneqq \Spec(R)$ is a simple $\delta$-$H_F$-torsor and $X \cong \Ind^{\psi}(Y)$ for some embedding $\psi \colon H\hookrightarrow G$. This proves parts \ref{induced_a}--\ref{induced_c}.
    
    By construction, the composite $Y \to X \to Z$ is $\rho$-equivariant, where $\rho\coloneqq \phi\circ \psi$. Therefore, $Z \cong \Ind^{\rho}(Y)$ by \Cref{prop: universal}. This proves \ref{induced_d}.
\end{proof}

It now makes sense to define to define $\delta$-Galois group of any simple $\delta$-torsors.

\subsection{Splitting a differential torsor}
\label{subsection:splitting}

We now study the splitting behavior of $\delta$-torsors. For the rest of this section, we assume any $\delta$-field extension $L$ of $F$ has the same field of constants as $F$.

\begin{definition}

Let $X$ be a $\delta$-$G_{F}$-torsor, $M$ a $\delta$-$F$-module, and $L$ a $\delta$-field extension of $F$. We say that $L$ is a \emph{splitting field} of $X$ (resp. $M$) or that $L$ \emph{splits} $X$ (resp. $M$) if $X_{L}$ is a trivial $\delta$-$G_{L}$-torsor (resp. $\delta$-$L$-module). A splitting field $L$ of $X$ (resp. $M$) is called \emph{minimal} if $\trdeg_F(L)$ is the smallest among the splitting fields of $X$ (resp. $M$). The \emph{splitting degree} of $X$ (resp. $M$), denoted $\deg_{\spl}^{\delta} (X)$ (resp. $\deg_{\spl}^{\delta} (M)$), is the transcendence degree of a minimal splitting field.
\end{definition}

We show that minimal splitting fields for $\delta$-modules and therefore for $\delta$-torsors is finite, as summarized below. 

\begin{lemma}\label{trivial_module}
    Let $M$ be a $\delta$-$F$-module. Then the Picard-Vessiot extension of $M$ is a minimal splitting field of $M$.
\end{lemma}
\begin{proof} Let $L$ be a splitting field of $M$.
    In an $F$-basis, $M$ corresponds to an equation $y'=Ay$ with $A \in \GL_n(F)$. Since $M\otimes_F L$ is trivial, $y'=Ay$ is gauge equivalent to $z'=0z$, so there exists a fundamental matrix $Y\in \GL_n(L)$ for $y'=Ay$. The Picard-Vessiot extension of $M$ over $F$ is $F(Y_{ij})\subseteq L$. 
\end{proof}

\begin{theorem}\label{prop:split_torsor} 
Let $X$ be a $\delta$-$G_F$-torsor and let $R$ be a simple $\delta$-quotient of $F[X]$. Then $K \coloneqq \Frac(R)$ is a minimal splitting field of $X$. Let $H$ be the $\delta$-Galois group for $R$. Then $\deg_{\spl}^{\delta}(X) = \dim(H)$. In particular, $\deg_{\spl}^{\delta}(X)$ is bounded by $\dim(G)$ for all $\delta$-$G$-torsor $X$.
\end{theorem}
\begin{proof}
    Fix an embedding $\phi\colon G \to \GL_n$. By \Cref{thm:comprehensive_induced_torsor}, $Z \coloneqq \Ind^{\phi}(X)$ is a $\delta$-$\GL_{n,F}$-torsor that corresponds to some $\delta$-module $M$ with Picard-Vessiot ring $R$. Moreover, $X = \Ind^{\psi}(Y)$ where $Y = \Spec (R)$ is a $\delta$-$H_F$-torsor and $\psi \colon H \to G$ is some embedding.

    We first show that $K$ is a splitting field of $X$. By the Torsor Theorem (\cite[Theorem 1.30]{vdPS03}), $K \otimes_F R \cong K \otimes_C C[H]$, i.e., $K$ splits $Y$. By \Cref{base_change_induced}, $X \times_F K = \Ind^{\psi}(Y) \times_F K \cong \Ind^{\psi}(Y \times_F K)$. Then 
    $$K[X] \cong (K \otimes_C C[H] \otimes_C C[G])^H \cong K \otimes_C C[G],$$ so $K$ splits $X$. Let $L$ be another $\delta$-field that splits $X$. Then by a similar argument, $Z \times_F L$ is trivial. Therefore, $L$ contains $K$ by \Cref{trivial_module}, whence $K$ is minimal. 

    Hence, $\deg_{\spl}^{\delta}(X) = \trdeg_F(K) = \dim(H)$ by \cite[Corollary 1.30]{vdPS03}. Moreover, since $H$ is a closed subgroup of $G$, $\dim(H) \leq \dim(G)$.
\end{proof}

\section{Twisted forms and descent}
\label{section:twisted-forms-descent}

This section defines twisted forms in the Picard-Vessiot theory and proves descent along $\delta$-Hopf-Galois extensions (\Cref{thm:descent}). Throughout this section, $R$ denotes a $\delta$-$C$-algebra, and all unadorned tensor products are taken over $R$, i.e., $\otimes = \otimes_{R}$. 

\subsection{Differential faithful flatness}

Since the category of $\delta$-modules over a $\delta$-ring $R$ is abelian, the notion of exact sequences is defined. 
\begin{definition}
    A $\delta$-ring homomorphism $R\to S$ is said to be \emph{$\delta$-faithfully flat} if the following holds: any given sequence $N_{\bullet}$ of $\delta$-$R$-modules is exact if and only if the sequence $N_{\bullet}\otimes_RS$ of $\delta$-$R$-modules is exact.
\end{definition}

When $R\to S$ is $\delta$-faithfully flat, the Amitsur complex $M\to M\otimes_R S\to M\otimes_R S^{\otimes 2}\to \cdots$ is exact for all $\delta$-$R$-module $M$ by similar arguments to \cite[Theorem III.6.6]{artin1999noncommutative}.

\begin{proposition}\label{prop:pv-ff}
Let $R$ be a simple $\delta$-ring (e.g., $R$ is a Picard-Vessiot ring). Then the inclusion $R\hookrightarrow \Frac(R)$ is a $\delta$-faithfully flat homomorphism.
\end{proposition}
\begin{proof}
    We follow \cite[Proposition 11.7]{milne2017primer}. 
    Let $K=\Frac(R)$. 
    Then a $\delta$-$R$-module $N$ is zero if and only if $N\otimes_R K$ is zero. 
    To see this, notice that when $N\neq 0$, we have $R\hookrightarrow N$ since $R$ is simple, and so $K\hookrightarrow N\otimes_R K$ since $K/R$ is flat.

    Now let $N_{\bullet} = (\,N'\overset{\alpha}{\to} N\overset{\beta}{\to} N''\,)$ be a sequence of $\delta$-$R$-modules. Since $K/R$ is flat, it suffices to check that $\beta \otimes 1_K$ is surjective. But this follows since
    $$
    N_{\bullet}\text{ is exact}
    \;\Leftrightarrow\;
    \frac{\ker\beta}{N'} = 0
    \;\Leftrightarrow\;
    \frac{\ker\beta}{N'}\otimes_R K = \frac{\ker(\beta\otimes 1)}{N'\otimes_RK} = 0
    \;\Leftrightarrow\;
    N_{\bullet}\otimes_R K\text{ is exact}.
    $$
\end{proof}

\subsection{Differential Hopf-Galois extension}\label{subsection:diff-hopf}
We first reformulate Picard-Vessiot rings as $\delta$-Hopf-Galois extensions. Let $S$ be a $\delta$-coalgebra over $R$ with comultiplication $\Delta$ and counit~$\epsilon$.

\begin{definition}
    A \emph{$\delta$-$S$-comodule} over $R$ is a $\delta$-$R$-module $M$ together with a $\delta$-$R$-linear map $\rho\colon M\to M\otimes_{R} S$ such that $(1\otimes \Delta)\circ \rho = (\rho\otimes 1)\circ \rho$ and $(1\otimes \epsilon)\circ \rho=1$.
\end{definition}

Let $H$ be a $\delta$-Hopf algebra over $R$ with comultiplication $\Delta_{H}$, counit $\epsilon_{H}$, and antipode $\sigma_{H}$. 

\begin{definition}
    Suppose that $S$ is a $\delta$-algebra over $R$ equipped with a map $\Delta_{S}\colon S\to S\otimes_{R} H$ such that $S$ is a $\delta$-$H$-comodule via the coaction map $\Delta_{S}$. If $M$ is both a $\delta$-$S$-module and a $\delta$-$H$-comodule with a $\delta$-$R$-linear map $\Delta_{M}\colon M\to M\otimes_{R} H$ satisfying $\Delta_{M}(ms)=\Delta_{M}(m)\Delta_{S}(s)$ for all $m\in M$ and all $s\in S$, we say that $M$ is a \emph{$\delta$-$(H,S)$-Hopf module} over $R$. The \emph{$H$-coinvariants} of $M$ is the $\delta$-$R$-submodule 
$$M^{\co H}=\{m\in M\;\vert\; \Delta_{M}(m) = m\otimes 1\}$$
of $M$. In the special case that $S=R$ and $\Delta_{S}(s) = s\otimes 1$ for all $s\in R$, we simply call $M$ a \emph{$\delta$-$H$-Hopf module} over $R$.
\end{definition}

\begin{definition}
    A \emph{$\delta$-$H$-Hopf-Galois extension} is a $\delta$-faithfully flat $\delta$-ring extension $S/R$ such that $S$ is a $\delta$-$H$-Hopf module over $R$, and such that the map
    \begin{equation}\label{eq:can-hopf}
        \begin{array}{rcc}
            \can_S\colon  S\otimes_{R} S&\to &S\otimes_{R} H\\
            x\otimes y& \mapsto & (x\otimes 1)\Delta_{S}(y)
        \end{array}
    \end{equation}
    is an isomorphism of $\delta$-algebras. 
\end{definition}

Let $R/F$ be a Picard-Vessiot ring extension with $\delta$-Galois group $G$. Since the notion of a $\delta$-Hopf-Galois extension is dual to that of a $\delta$-torsor, $R/F$ is a $\delta$-$F[G]$-Hopf-Galois extension.

\subsection{\texorpdfstring{$\bphi$}{Phi}-structures}\label{subsection:phi-structure}
We can formalize the ``differential objects'' ($\delta$-module, $\delta$-$R$-algebra, Picard-Vessiot ring, etc.) from the last section as $\bphi$-objects.
We roughly follow the formalism in \cite[Section 1.3]{nardin2012essential}. 

\begin{definition}
    Let $R$ be a $C$-algebra. A \emph{type} over $R$ is a triple 
\[
    \bphi = \left(H, I, \{(n_{1i}, n_{2i},n_{3i},n_{4i})\}_{i\in I}\right)
\]
where $H$ is a $\delta$-Hopf algebra over $R$, $I$ is a set, and $\{(n_{1i}, n_{2i},n_{3i},n_{4i})\}_{i\in I}$ is a subset of $\mathbb{N}^4$. Given such a type $\bphi$ over $R$, a \emph{$\bphi$-object}  $(M, \{\Phi_i\})$ is a $\delta$-$R$-module $M$ equipped with a collection of $\delta$-$R$-module homomorphisms
\[
    \left\lbrace
    \Phi_i\colon M^{\otimes n_{1i}}\otimes H^{\otimes n_{2i}}\to M^{\otimes n_{3i}}\otimes H^{\otimes n_{4i}}
    \right\rbrace_{i\in I}.
\] A \emph{morphism} of $\bphi$-objects $(M,\{\Phi_i\}_{i\in I})$ and $(N,\{\Psi_i\}_{i\in I})$ is a $\delta$-$R$-module homomorphism $\varphi\colon M\to N$ that makes the following diagram commute for all $i\in I$.
    \[
    \begin{tikzcd}
        M^{\otimes n_{1i}}\otimes H^{\otimes n_{2i}}\ar[rr,"\Phi_{i}"] \ar[d,"\varphi^{\otimes n_{1i}}\otimes \id"']&& M^{\otimes n_{3i}}\otimes H^{\otimes n_{4i}}\ar[d,"\varphi^{\otimes n_{3i}}\otimes \id"]\\
        N^{\otimes n_{1i}}\otimes H^{\otimes n_{2i}}\ar[rr,"\Psi_{i}"] && N^{\otimes n_{3i}}\otimes H^{\otimes n_{4i}}
    \end{tikzcd}
    \]
    The $\bphi$-objects and their morphisms form a category which we denote by $\bphi\Object$. In particular, an automorphism of a $\bphi$-object $M$ is a $\delta$-$R$-module automorphism of $M$ that preserves the type $\bphi$. The group of all automorphisms of $M$ is denoted by $\Aut^{\delta, \bphi}_R(M)$.
\end{definition}

\begin{definition}
    Let $S$ be a $\delta$-$R$-algebra and let $\bphi=\left(H, I, \{\Phi_{i}\}_{i\in I}\right)$ be a type over $R$. Then 
\[
    \bphi_S\coloneqq \left(H\otimes_RS, I, \{\Phi_{i}\otimes \id_{S}\}_{i\in I}\right)
\] 
is a type over $S$, and extension of scalars $M\mapsto M\otimes S$ takes a $\bphi$-object to a $\bphi_S$-object.
Let $\bphi$ be a type over $C$. A $\bphi_R$-object $M$ is \emph{trivial} if there exists a $\bphi$-object $M^{\delta}$ such that $M = M^{\delta}\otimes_C R$.
\end{definition}

\begin{example}
    We will always view 
\begin{enumerate}
    \item a $\delta$-$R$-module as a $\bphi$-object of type $(R, \varnothing, \varnothing)$;
    \item a $\delta$-$R$-algebra $A$ with multiplication $m\colon A^{\otimes 2}\to A$ as a $\bphi$-object of type $(R, \{1\}, \{(2, 0, 1, 0)\})$;
    \item a $\delta$-$H$-Hopf-Galois extension $S/R$ with multiplication $m\colon S^{\otimes 2}\to S$ and coaction $\Delta_{S}\colon S\to S\otimes H$ as a $\bphi$-object of type $(H, \{1,2\}, \{(2, 0, 1, 0), (1,0,1,1)\})$.
\end{enumerate}
\end{example}

\subsection{Equivariant \texorpdfstring{$\bphi$}{Phi}-structures}
Let $\bphi$ be a type over $R$ and let $H'$ be a $\delta$-Hopf algebra over $R$. 

\begin{definition}
    An \emph{$H'$-equivariant $\bphi$-object} is a $\bphi$-object $(M,\{\Phi_i\}_{i\in I})$ such that $M$ is a $\delta$-$H'$-comodule with coaction map $\Delta_{M}\colon M\to M\otimes H'$ making the following diagram commute
\[
\begin{tikzcd}
    M^{\otimes n_{1i}}\otimes H^{\otimes n_{2i}}\ar[rr,"\Phi_{i}"] \ar[d,"\Delta_{M}^{\otimes n_{1i}}\otimes \id"']&& M^{\otimes n_{3i}}\otimes H^{\otimes n_{4i}}\ar[d,"\Delta_{M}^{\otimes n_{3i}}\otimes \id"]\\
    M^{\otimes n_{1i}}\otimes H^{\otimes n_{2i}}\otimes H'\ar[rr,"\Phi_{i}\otimes1_{H'}"'] && M^{\otimes n_{3i}}\otimes H^{\otimes n_{4i}}\otimes H'
\end{tikzcd}
\]
for all $i\in I$. A \emph{morphism of $H'$-equivariant $\bphi$-objects} $(M,\{\Phi_i\}_{i\in I})$ and $(N,\{\Psi_i\}_{i\in I})$ is a morphism $\varphi\colon M\to N$ of $\bphi$-objects making the following diagram commute.
\[
\begin{tikzcd}
M \ar[d,"\Delta_{M}"']\ar[r,"\varphi"]& N\ar[d,"\Delta_{N}"]\\
M\otimes H'\ar[r,"\varphi\otimes 1_{H'}"'] & N\otimes H' 
\end{tikzcd}
\] We denote the category of $H'$-equivariant $\bphi$-objects and their morphisms by $\bphi^{H'}\!\Object$.
\end{definition}

\subsection{Descent}

Let $\bphi=\left(H, I, \{\Phi_{i}\}_{i\in I}\right)$ be a type over $R$, and $S/R$ a $\delta$-$H'$-Hopf-Galois extension. Given a $\bphi$-object $(M, \{\Phi_i\}_{i\in I})$, the $\bphi_{S}$-object $(M\otimes_{R}S, \{\Phi_i\otimes \id_S\}_{i\in I})$ is $H'$-equivariant over $S$ since $\Phi_{i}\otimes \id_S$ commute with $\Delta_{S}$ for all $i \in I$. This defines a functor
\begin{equation}\label{extend-scalar-2}
    \bphi\Object\to \bphi_{S}^{H'}\!\Object
\end{equation}

Conversely, given a $H'$-equivariant $\bphi_S$-object $(N,\{\Phi_{i}\})$ over $S$, we may consider its coinvariant module $N^{\co H'} = \{n\in N\;|\; \Delta_{N}(n) = n\otimes 1\}$. Since each $\Phi_i$ commutes with $\Delta_{N}$, the $\bphi_S$-structure on $N$ restricts to a $\bphi$-object on $N^{\co H'}$. This gives a functor 
\[
    \bphi_{S}^{H'}\!\Object\to \bphi\Object
\]

We now show that the two functors define an equivalence of categories. We follow the proof of \cite[Theorem 3.7 (1) $\Rightarrow$ (2)]{schneider1990principal}.

\begin{theorem}[Descent along $\delta$-Hopf-Galois extensions]\label{thm:descent}
Let $S/R$ be a $\delta$-$H'$-Hopf-Galois extension and let $\bphi = \left(H, I, \{\Phi_{i}\}_{i\in I}\right)$ be a type over $R$. Then extension of scalars defines an equivalence of categories
$$
\bphi\Object\to \bphi_{S}^{H'}\!\Object.
$$
The pseudo-inverse is given by taking an object $N$ to the coinvariant module $N^{\co H'}$.
\end{theorem}
\begin{proof}
Naturality is clear. We are left to check that the maps 
\[
\begin{array}{rcc}
\mu_N\colon  N^{\co H'}\otimes_{R}S& \to & N\\
n\otimes s&\mapsto & ns
\end{array}\quad \text{ and } \quad 
\begin{array}{rcc}
\iota_{M}\colon M& \to & (M\otimes_{R}S)^{\co H'}\\
m &\mapsto & m\otimes 1
\end{array}
\]
are bijections for all $M$ in $\bphi\Object$ and $N$ in $\bphi_{S}^{H'}\!\Object$.

Consider the following two commutative diagrams.
\begin{equation}\label{coinvar-1}
\begin{tikzcd}
N^{\co H'}\otimes_{R}S\ar[r]\ar[dd,"\mu_{N}"']  
& N\otimes_{R}S\ar[dd,"\id_N\underset{S}{\otimes}\can_{S}"]\ar[r,shift left=1.5]\ar[r,shift right=1.5]  
& (N\otimes_{S} H')\otimes_{R} S\ar[dd,"\id_{N\otimes H'}\underset{S}{\otimes}\can_{S}"]\\
\\
N\ar[r,"\Delta_{N}"']
& N\otimes_{S}H'\ar[r,shift left=1.5,"\Delta_{N\otimes H'}"]\ar[r,shift right=1.5,"\id_N\otimes \Delta_{H'}"'] 
& N\otimes_{S}H'\otimes_{S}H'\\
\end{tikzcd}
\end{equation}

\begin{equation}\label{coinvar-2}
\begin{tikzcd}
(M\otimes_{R} S)^{\co H'}\ar[r,"\iota"]
& M\otimes_{R}S\ar[r,shift left=1.5]\ar[r,shift right=1.5] 
& M\otimes_{R}S\otimes_{R}H'\\\\
M\ar[r,"\iota"']\ar[uu,"\iota_{M}"]  
& M\otimes_{R}S\ar[uu,equal]\ar[r,shift left=1.5,"\iota_{1}"]\ar[r,shift right=1.5,"\iota_{2}"']  
& M\otimes_{R} S\otimes_{R} S\ar[uu,"\id_M\otimes\,\can_{S}"']
\end{tikzcd}
\end{equation}
The vertical maps $\id_{N}\otimes_{S}\can_{S}$, $\id_{N\otimes H'}\otimes_{S}\can_{S}$, and $\id_M\otimes \can_{S}$ are isomorphisms since \eqref{eq:can-hopf} is an isomorphism.
The top rows of \eqref{coinvar-1} and \eqref{coinvar-2} are exact by the definition of $\co H'$ and also the $\delta$-flatness of $S/R$ in the case of \eqref{coinvar-1}. 
The bottom row of \eqref{coinvar-1} is exact by coassociativity of $\Delta_{N}$. Since $S/R$ is $\delta$-faithfully flat, the Amitsur complex in the bottom row of \eqref{coinvar-2} is exact. 
So $\mu_{N}$ and $\iota_{M}$ are isomorphisms.
\end{proof}

\subsection{Twisted forms}\label{subsection:tf}
Oftentimes, a given type of $\bphi$-objects is too varied for direct classification. Instead, we view two $\bphi$-objects as related if they become isomorphic after extension of scalars. 

\begin{definition}
    Let $S$ be a $\delta$-$R$-algebra, $\bphi$ be a type over $R$, and $M$ and $N$ be $\bphi$-objects. We call $N$ an \emph{$(S/R)$-twisted form} of $M$ if there exists an isomorphism $N\otimes_{R} S\cong M\otimes_{R} S$ of $\bphi_S$-objects. We let $\TF(S/R,M)$ denote the set of isomorphism classes of $(S/R)$-twisted forms of $M$.
\end{definition}

\begin{example}\label{ex-twisted-forms}\hfill
\begin{enumerate}
    \item Any $\delta$-module over $F$ is a $(F^{\PV}/F)$-twisted form of the trivial $\delta$-module $M$ of the same rank. Thus $\diffmod_{n}(F) = \TF(F^{\PV}/F, M)$.
    \item Let $H$ be a $\delta$-Hopf algebra over $R$. Let $S/R$ be a $\delta$-$H$-Hopf-Galois extension. Then $S$ is an $(S/R)$-twisted form of $H$ via the  isomorphism $\can_{S}$.
\end{enumerate}
\end{example}

$\delta$-torsors are another important example of twisted forms. Given a $\delta$-$R$-algebra $S$ and a linear algebraic group $G$ over $C$, we say that a $\delta$-$G_{R}$-torsor $X$ is an \emph{$(S/R)$-twisted form} of a $\delta$-$G_{R}$-torsor $Y$ if $X_{S}$ is isomorphic to $Y_{S}$ as $\delta$-$G_{S}$-torsors; equivalently, $R[X]$ is an $(S/R)$-twisted form of $R[Y]$ as $\delta$-$R[G]$-Hopf-Galois extensions. Given a $\delta$-$G_{R}$-torsor $X$, we let $\TF(S/R, X)$ denote the set of  isomorphism classes of $(S/R)$-twisted forms of $X$.

\begin{example}\label{ex:tors-twisted-forms}
Let $G$ be a linear algebraic group over $C$. By \Cref{prop:split_torsor}, any $\delta$-$G_F$-torsor is a $(F^{\PV}/F)$-twisted form of the trivial $\delta$-$G_F$-torsor $G_{F}$. In particular, $\delta\text{-}G_{F}\tors(F)= \TF(F^{\PV}/F, G_{F})$.
\end{example}

Finally, we compute the automorphism group of trivial $\bphi$-objects. Note that the following lemma already appears in \cite[Proposition 3.8]{masuoka2021twisted}. We present it according to our framework. 

\begin{lemma}
\label{lemma:automorphism_split}
    Let $\bphi$ be a type over $C$. Let $M$ be a trivial $\bphi_F$-object, i.e., there exists a $\bphi$-object $M^{\delta}$ such that $M = M^{\delta} \otimes_C F$. Then $\Aut^{\delta, \bphi}_R(M \otimes_F R) = \Aut_{R^{\delta}}(M^{\delta} \otimes_C R^{\delta})$ for any $\delta$-$F$-algebra $R$. In particular, if the functor 
    $$\underline{\Aut}_C(M^{\delta}) \colon (C\operatorname{-Algebras}) \to (\operatorname{Groups})\colon
        D \mapsto \Aut_D(M^{\delta} \otimes_C D)$$ is represented by a linear algebraic group $G$ over $C$, then $\Aut_F^{\delta, \bphi}(M) = G(C).$
\end{lemma}

\begin{proof}
    First an easy computation shows that $(M \otimes_F R)^{\delta} = (M^{\delta} \otimes_C R)^{\delta} = M^{\delta} \otimes_C R^{\delta}$.

    An automorphism $\phi \in\Aut^{\delta, \bphi}_R(M \otimes_F R)$  preserves constants, so $\phi$ and $\phi^{-1}$ restrict to $R^{\delta}$-module automorphisms $\psi, \psi^{-1}$ on $M^{\delta} \otimes_C R^{\delta}$. Since $\phi\circ\phi^{-1} = \phi^{-1}\circ\phi = \id$
    on $M \otimes_F R$, $\phi$ and $\phi^{-1}$ are still inverses on $M^{\delta} \otimes_C R^{\delta}$. Moreover, extending $\psi$ $R$-linearly, $\psi \otimes_{R^{\delta}} \id_R$ recovers $\phi$. Hence, 
    \begin{equation*}
        \Aut^{\delta, \bphi}_{R}(M \otimes_F R) = \Aut_{R^{\delta}}(M^{\delta} \otimes_C R^{\delta}). 
    \end{equation*}  In particular, if $\underline{\Aut}_C(M^{\delta})$ is represented by a linear algebraic group $G$ over $C$, $\Aut^{\delta, \bphi}_F(M) = G(C)$ by taking $R = F$.
\end{proof}

This lemma in particular tells us that for a trivial $\delta$-$G_F$-torsor $X$, $\Daut_F(X) = G(C)$.

\section{Cohomology}\label{section:cohom}

We now define the cohomology that classifies the twisted forms of the previous section.

\begin{definition}\label{definition:cohom}
Let $\Gamma$ and $G$ be linear algebraic groups over $C$. Let $\Gamma$ act on $G$ (as varieties) in a way that is compatible with the group structure on $G$. A \emph{1-cocycle} is a morphism of varieties $a\colon \Gamma(C)\to G(C)$ satisfying the following condition: for any $\sigma\in \Gamma(C)$ we let $a_{\sigma}\coloneqq a(\sigma)$ and require that $a_{\sigma\tau} = a_{\sigma}\cdot \sigma(a_{\tau})$ holds for all $\sigma,\tau\in \Gamma(C)$. Two 1-cocycles $a$ and $b$ are \emph{equivalent} if there exists $c\in G(C)$ such that $a_{\sigma}=c\cdot b_{\sigma}\cdot c^{-1}$. We define $H^{1}_{\delta}(\Gamma,G)$ to be the set of 1-cocycles $\Gamma\to G$ modulo equivalence. 
\end{definition}

 We will adopt the following setup.

\begin{assumption}\label{assumption}
    Let $\Gamma$ be a linear algebraic group over $C$ with Hopf algebra $H'_0$ and let $S/R$ be a $\delta$-$(H'_0\otimes_{C}R)$-Hopf-Galois extension. Let $(M,\{\Phi_{i}\}_{i\in I})$ be a $\bphi$-object and set $M_S\coloneqq M\otimes_{R}S$. Furthermore assume that the automorphism group of $M$ is represented by a linear algebraic group $G$ over $C$, i.e., there exists an isomorphism
    \begin{equation}\label{repble}
        \Aut^{\delta, \bphi}(M_S\otimes_{C}D)\cong G(D)
    \end{equation}
that holds for every $C$-algebra $D$ and is functorial in $D$. We will always identify the two groups in \eqref{repble}.
\end{assumption}

\begin{remark}\label{conversion}
Note that \eqref{repble} gives a sequence of isomorphisms
\begin{eqnarray}
    \Aut^{\delta, \bphi}(M_S\otimes_{C}{H'_{0}}^{\otimes n})&\cong& G({H'_{0}}^{\otimes n})\\
    &\cong &\Hom_{C-\Alg}(C[G], {H'_{0}}^{\otimes n})
    \cong \Mor_{C-\text{Sch}}(\Gamma^{n}, G)\nonumber
\end{eqnarray}
which takes an element $f\in \Aut(M_S\otimes_{C}{H'_{0}}^{\otimes n})$ to the morphism $\Gamma^{n}\to G$ given on $C$-points by
\begin{equation}
    \begin{array}{rl}
\Gamma^{n}(C)&\to G(C)\\
(\sigma_{1},...,\sigma_{n})&\mapsto (\id_{M_S}\otimes \sigma_{1}\otimes \cdots \sigma_{n})_{*}f.
\end{array}
\end{equation}
Here $(\id_{M_S}\otimes \sigma_{1}\otimes \cdots \sigma_{n})_{*}f$ is the morphism obtained by extension of scalars such that the following diagram commutes:
\begin{equation}\label{eq:extend-scalars}
\begin{tikzcd}
M_S\otimes_{C} {H'_{0}}^{\otimes n}\ar[rr,"f"]\ar[d, "\id_{M_S}\otimes \sigma_{1}\otimes \cdots \otimes\sigma_{n}"']&& M_S\otimes_{C} {H'_{0}}^{\otimes n}\ar[d, "1_{M_S}\otimes \sigma_{1}\otimes \cdots \otimes\sigma_{n}"]\\
M_S\ar[rr,"(\id_{M_S}\otimes \sigma_{1}\otimes \cdots \otimes\sigma_{n})_*f"']&&M_S.
\end{tikzcd}
\end{equation}
\end{remark}

\begin{remark}
\Cref{assumption} gives an action of $\Gamma$ on $G$ that we now describe. First we define an action of $\Gamma(C)$ on $S$ by letting an element $\sigma\in \Gamma(C)$ act on $S$ via the automorphism $\boldsymbol{\sigma} = (1_{S}\otimes \sigma)\circ \Delta_{S}$:
\begin{equation}
\boldsymbol{\sigma}\colon S\xrightarrow{\Delta_{S}}S\otimes_{C}{H'_{0}}\xrightarrow{1_{S}\otimes \sigma}S.
\end{equation}
This action extends to an action of $\Gamma(C)$ on $M_S \coloneqq M\otimes_{R}S$. The action of $\Gamma(C)$ on $M_S$ further gives an action of $\Gamma$ on $G$ on the $C$-points by conjugation: for any $\sigma\in \Gamma(C)$ and $\varphi\in G(C)$, we define the action to be $\sigma(\varphi)\coloneqq \boldsymbol{\sigma}\circ \varphi\circ\boldsymbol{\sigma}^{-1}$. We will consider $G$ with this $\Gamma$-action when discussing the cohomology set $H^{1}_{\delta}(\Gamma,G)$.
\end{remark}

\begin{remark}

In the case $n=1$ in \Cref{conversion}, an element $a\in \Aut(M_S\otimes_{C}{H'_{0}})$ corresponds to an element of $\Mor_{C-\text{Sch}}(\Gamma,G)$ which we again denote by $a$. If for each $\sigma\in \Gamma(C)$ we let $a_{\sigma}\coloneqq a(\sigma)$, we have the equality 
$$
(\id_{M_S}\otimes\sigma)_{*}a = a_{\sigma}.
$$
The commutativity of the diagram
$$
\begin{tikzcd}
M_S\ar[r,"\Delta_{M_S}"]\ar[rd,"\boldsymbol{\sigma}"']& M_S\otimes_{C}{H'_{0}}\ar[r,"a"]\ar[d,"\id\otimes \sigma"]& M_S\otimes_{C}{H'_{0}}\ar[d,"\id\otimes \sigma"]\\
& M_S\ar[r,"a_{\sigma}"']& M_S
\end{tikzcd}
$$
further gives the equality
\begin{equation}\label{convert-important}
(\id_{M_S}\otimes\sigma)\circ(a\circ \Delta_{M_S})=a_{\sigma}\circ \boldsymbol{\sigma}
\end{equation}
which we will later use in \Cref{lemma:well-defined-G}.

\end{remark}

\begin{construction}\label{construction:F}
We define a map 
$$
\mathcal F\colon \TF(S/R,M)\to H^{1}_{\delta}(\Gamma,G)
$$
as follows. Let $(N,\varphi)$ be a twisted form of $M$. We define $\mathcal F(N,\varphi)$ to be the cocycle $a\colon \Gamma(C)\to G(C)$ which sends an element $\sigma$ in $\Gamma(C)$ to the element 
\[
    a_{\sigma}\varphi\circ \boldsymbol{\sigma}(\varphi)= \varphi\circ \boldsymbol{\sigma}\circ\varphi^{-1}\circ \boldsymbol{\sigma}^{-1}\text{ in }G(C).
\]

We must verify that $\mathcal F$ is well-defined. That $a$ is a cocycle follows from the standard computation in $G(C)$:
$$
\begin{array}{rl}
a_{\sigma} \cdot \sigma(a_{\tau})    & =a_{\sigma} \circ \boldsymbol{\sigma}\circ a_{\tau}\circ \boldsymbol{\sigma}^{-1} \\
     & = (\varphi\circ \boldsymbol{\sigma}\circ\varphi^{-1}\circ\boldsymbol{\sigma}^{-1}) \circ \boldsymbol{\sigma} \circ (\varphi\circ \boldsymbol{\tau}\circ\varphi^{-1}\circ\boldsymbol{\tau}^{-1})\circ \boldsymbol{\sigma}^{-1} \\
     & = \varphi\circ (\boldsymbol{\sigma}\circ\boldsymbol{\tau})\circ\varphi^{-1}\circ(\boldsymbol{\sigma}\circ \boldsymbol{\tau})^{-1}\\
     & = a_{\sigma\tau}
\end{array}
$$
which holds for all $\sigma,\tau\in \Gamma(C)$.

Next let $(N',\psi)$ be isomorphic to $(N,\varphi)$ as twisted forms of $M$ and let $b = \mathcal F(N',\psi)$. Setting $c = \psi\circ \varphi^{-1}$, we have
$$
\begin{array}{rl}
   c^{-1}\circ b_{\sigma}\circ \sigma(c)  &= c^{-1}\circ b_{\sigma}\circ \boldsymbol{\sigma}\circ c\circ \boldsymbol{\sigma}^{-1}  \\
   & = (\varphi\circ \psi^{-1})\circ (\psi\circ \boldsymbol{\sigma}\circ\psi^{-1}\circ\boldsymbol{\sigma}^{-1})\circ \boldsymbol{\sigma}\circ (\psi\circ \varphi^{-1})\circ \boldsymbol{\sigma}^{-1}\\
    & = \varphi\circ \boldsymbol{\sigma}\circ\varphi^{-1}\circ\boldsymbol{\sigma}^{-1} \\
     &  = a_{\sigma}
\end{array}
$$
for all $\sigma\in \Gamma(C)$. Thus $\mathcal F$ takes equivalent twisted forms to equivalent cocycles. We conclude that $\mathcal F$ is well-defined.
\end{construction}

\begin{construction}\label{construction:G}
We define a map \begin{equation}
\mathcal G:H^{1}_{\delta}(\Gamma,G)\to \TF(S/R,M)
\end{equation}
as follows. Given a cocycle $a$ representing an element of $H^{1}_{\delta}(\Gamma,G)$, we define the $\delta$-$R$-module
\begin{equation}
\mathcal G(a) \coloneqq  \left\lbrace m\in M_S\;\left\vert\; (a_{\sigma}\circ \boldsymbol{\sigma})(m)=m \text{ for all }\sigma\in \Gamma(C)\right.\right\rbrace.
\end{equation}
\end{construction}

\bigskip

We will soon check that $\mathcal G$ is a well-defined map (in \Cref{lemma:well-defined-G}). Our proofs of \Cref{lemma:well-defined-G} and \Cref{thm:F-G-inverse} below follow that of \cite[Theorem 2.6]{nuss2007non} where a cohomology set was introduced to classify Hopf-Galois extensions for noncommutative rings. The following lemma allows us to convert from their ``cochains'' which are maps $M_S\to M_S\otimes_{C} {H'_{0}}$ to our cochains which are morphisms $\Gamma\to G$.

\begin{lemma}\label{lemma:equality}
Let $M$ be a $C$-vector space, $X$ an algebraic variety over $C$, and $f,g\in M\otimes_{C}C[X]$. If $(1_{M}\otimes \sigma)f=(1_{M}\otimes \sigma)g$ for all $\sigma\in X(C)$ then $f=g$ in $M\otimes_{C}C[X]$.
\end{lemma}
\begin{proof}
Since $X$ is an algebraic variety over an algebraically closed field $C$, equality of functions on $X(C)$ implies equality of elements in $C[X]$. This gives the case $M=C$.

For a general $M$, let $\{m_{i}\}_{i\in I}$ be a basis of $M$ over $C$ and write $f=\sum m_{i}\otimes f_{i}$ and $g=\sum m_{i}\otimes g_{i}$ for some $f_{i},g_{i}\in C[X]$. For all $\sigma\in X(C)$, we have $(1_{M}\otimes \sigma)(f) = (1_{M}\otimes \sigma)(g)$ hence $\sum m_{i}\otimes \sigma(f_{i}) = \sum m_{i}\otimes \sigma(g_{i})$. The linear independence of $\{m_{i}\}_{i\in I}$ over $C$ and the previous paragraph now give $f_i=g_{i}$ for all $i\in I$.
\end{proof}

\begin{lemma}\label{lemma:well-defined-G}
The map $\mathcal G$ in \Cref{construction:G} is well-defined.
\end{lemma}

\begin{proof} \emph{Step 1.}
We first check that given a cocycle $a$, $\mathcal G(a)$ is a twisted form of $M$. First note that $a$ satisfies the following properties:
\begin{enumerate}
    \item[(a)] $(a_{\sigma}\circ\boldsymbol{\sigma})(ms) = (a_{\sigma}\circ \boldsymbol{\sigma})(m)\boldsymbol{\sigma}(s)$ for all $\sigma\in \Gamma(C)$, $m\in M_S$, and $s\in S$;
    \item[(b)] $a_{1} = 1_{M_S}$;
    \item[(c)] $a_{\sigma\tau}=a_{\sigma}\circ \boldsymbol{\sigma}\circ a_{\tau}\circ \boldsymbol{\sigma}^{-1}$ for all $\sigma, \tau\in \Gamma(C)$.
\end{enumerate}
Here (c) is the cocycle condition for $a$, (b) follows from (c) by letting $\sigma=\tau=1$ in ${\Gamma(C)}$, and (a) follows from the $S$-linearity of $a_{\sigma}$. 

We claim that the composite map
$\Delta'$ given by $a\circ \Delta_{M_S}\colon M_S\to M_S\otimes_C {H'_{0}}$
defines a coaction on $M_S$ making $M_S$ a $\delta$-$({H'_{0}},S)$-Hopf module. In other words, we must verify the following properties:
\begin{enumerate}
    \item[(A)] $\Delta'(ms)=\Delta'(m)\Delta(s)$ for all $m\in M_S, s\in S$;
    \item[(B)] $(1_{M_S}\otimes \epsilon_{{H'_{0}}})\circ \Delta' = 1_{M_S}$;
    \item[(C)] $(\Delta'\otimes 1_{{H'_{0}}})\circ \Delta' = (1_{M_S}\otimes \Delta_{{H'_{0}}})\circ \Delta'$.
\end{enumerate}
\bigskip
Since $\epsilon_{{H'_{0}}}\colon {H'_{0}}\to C$ corresponds to $1\in \Gamma(C)$, (B) follows from (b) by \eqref{convert-important}.
\bigskip

To show (A) and (C), by \Cref{lemma:equality}, it suffices to show that the equalities obtained by applying $(1_{M_S}\otimes \sigma)$ to (A) and $(1_{M_S}\otimes \sigma\otimes \tau)$ to (C) hold for all $\sigma,\tau\in \Gamma(C)$. Applying $(1_{M_S}\otimes \sigma)$ to (A) and simplifying by \eqref{convert-important} gives (a). Thus (A) holds. Similarly, applying $1\otimes (\sigma \circ\tau)$ to the right side of (C) gives
\begin{equation}\label{eq:left}
    (1_{M_S}\otimes \sigma \otimes\tau)\circ (1_{M_S}\otimes \Delta_{{H'_{0}}})\circ \Delta' = (1_{M_S}\otimes (\sigma \circ\tau))\circ \Delta' = a_{\sigma\tau}\circ (\boldsymbol{\sigma\circ\tau}).
\end{equation}
Applying $1_{M_S}\otimes (\sigma \circ\tau)$ to the left side of (C) gives
\begin{equation}\label{eq:right}
    a_{\sigma}\circ \boldsymbol{\sigma}\circ a_{\tau}\circ \boldsymbol{\sigma}^{-1}
\end{equation}
since the following diagram commutes:
$$\begin{tikzcd}
M_S\ar[d,"\Delta'"']\ar[rrd,"a_{\tau}\circ \boldsymbol{\tau}"]   &    &      \\
M_S\otimes_{C} {H'_{0}} \ar[d,"\Delta'\otimes 1_{{H'_{0}}}"'] \ar[rr,"1\otimes\tau"]\ar[rd,"(a_{\sigma}\circ \boldsymbol{\sigma})\otimes 1_{{H'_{0}}}"]          &       &     M_S\ar[d,"a_{\sigma}\circ \boldsymbol{\sigma}"]  \\
M_S\otimes_{C} {H'_{0}}\otimes_{C} {H'_{0}}\ar[r,"1\otimes\sigma\otimes 1"'] &M_S\otimes_C {H'_{0}}\ar[r,"1\otimes\tau"']\ar[ru,phantom, "(I)"] & M_S
\end{tikzcd}.
$$
Here, region $(I)$ commutes by the computation
\begin{align*}
    ((a_{\sigma}\circ \boldsymbol{\sigma})\circ (1\otimes\tau))(m\otimes h) & = ((a_{\sigma}\circ \boldsymbol{\sigma})(m\cdot\tau(h))\\
    & = ((a_{\sigma}\circ \boldsymbol{\sigma})(m))\cdot\tau(h)\\
    & = ((1\otimes\tau)\sigma(a_{\sigma}\circ \boldsymbol{\sigma}))(m\otimes h).
\end{align*}

Equating \eqref{eq:left} with \eqref{eq:right} gives (c). Thus (C) holds. This concludes checking that $\Delta'$ defines a coaction on $M_S$ making $M_S$ a $\delta$-$({H'_{0}},S)$-Hopf module.

\Cref{thm:descent} implies that the coinvariant module $(M_S)^{\co \Delta'}$ is a twisted form of $M$ over $R$. Therefore it suffices to show that $\mathcal{G}(a)$ equals $(M_S)^{\co \Delta'}$. Any $m\in (M_S)^{\co \Delta'}$ satisfies $\Delta'(m)=m\otimes 1$. For any $\sigma\in \Gamma(C)$, applying $(1\otimes\sigma)$ to $\Delta'(m)=m\otimes 1$ and simplifying by \eqref{convert-important} gives $(a_{\sigma}\circ \sigma)(m) = m$. Thus $(M_S)^{\co \Delta'} \subseteq \mathcal{G}(a)$. Invoking \Cref{lemma:equality} gives the reverse inclusion and so $(M_S)^{\co \Delta'} = \mathcal{G}(a)$.

\emph{Step 2. The map $\mathcal G$ takes equivalent cocycles to isomorphic twisted forms.} Two cocycles $a$ and $b$ are equivalent means that $b = c\circ a\circ c^{-1}$ for some $c\in G(C) = \Aut^{\delta, \bphi}(M_S)$. The automorphism $c\colon M_S\to M_S$ restricts to an isomorphism $\mathcal{G}(a)\cong \mathcal{G}(b)$.

\end{proof}

Here is the main theorem of this section.

\begin{theorem}\label{thm:F-G-inverse}
Consider the above setup. The maps $\mathcal F$ and $\mathcal G$ are inverses. Hence there is a bijection between the two sets $\TF(S/R,M)$ and $H^{1}_{\delta}(\Gamma,G)$.
\end{theorem}
\begin{proof}
We first check $\mathcal G\circ \mathcal F = 1$. Let $(N,\varphi)$ be a twisted form of $M_S$ with associated cocycle $a = \mathcal F(N,\varphi)$. Set $P\coloneqq \mathcal G(a)$. We want to show $P \cong N$. Clearly the isomorphism $\varphi\colon N\otimes S\to M_S$ has image in $P$, so $\varphi$ restricts to $\varphi\vert_{N}\colon N\to P$. Consider the commutative diagram $$
\begin{tikzcd}
N\otimes S\ar[rr,"\varphi\vert_{N}\otimes 1_{S}"]\ar[rrd,"\varphi"',"\cong"]&&P\otimes S\ar[d,"\text{mult.}\cong"]\\
&& M_S.
\end{tikzcd}
$$
By definition of twisted form, the vertical map is an isomorphism so $\varphi\vert_{N}\otimes 1_{S}$ is too. By $\delta$-faithful flatness of $S/R$, $\varphi\vert_{N}$ is an isomorphism.

We now check $\mathcal F\circ \mathcal G = 1$. Given a cocycle $a$, let $(N,\varphi) \coloneqq  \mathcal G(a)$ and $b\coloneqq  \mathcal F(N,\varphi)$. Given $\sigma\in \Gamma(C)$, consider the diagram
$$
\begin{tikzcd}
M_S\ar[rrr,"b_{\sigma}\circ \boldsymbol{\sigma}"]\ar[dd,equal] & & & M_S\ar[dd,equal]\\
 & N\ar[ul,"\varphi"']\ar[dl,"\varphi"]\ar[r,"\boldsymbol{\sigma}"] & N\ar[ur,"\varphi"]\ar[dr,"\varphi"']& \\
M_S\ar[rrr,"a_{\sigma}\circ \boldsymbol{\sigma}"'] & & & M_S
\end{tikzcd}.
$$
The triangles here trivially commute. The upper trapezoid commutes by definition of $b_{\sigma}$. The bottom trapezoid commutes by the following two computations. For all $n\in N$ and $s\in S$, we have
\[
    \varphi(\boldsymbol{\sigma}(n\otimes s)) = \varphi((n\otimes 1)(1\otimes \sigma(s))) = n\sigma(s)
\]
where the second equality uses the isomorphism $N\otimes S\cong M_S$ given by scalar multiplication. Also
\[
    (a_{\sigma}\circ \boldsymbol{\sigma}\circ\varphi)(n\otimes s)
    = (a_{\sigma}\circ \boldsymbol{\sigma})(ns)
    = (a_{\sigma}\circ \boldsymbol{\sigma})(n)\sigma(s)
    = n\sigma(s)
\]
where the second equality uses property (a) in the proof of \Cref{lemma:well-defined-G} and the third equality uses the definition of $N$. Therefore the diagram above is commutative so $b_{\sigma}\circ \boldsymbol{\sigma} = a_{\sigma}\circ \boldsymbol{\sigma}$ for all $\sigma\in \Gamma(C)$. By \eqref{convert-important} and \Cref{lemma:equality}, we have $b = a$.
\end{proof}

\subsection{Absolute cohomology}
While \Cref{thm:F-G-inverse} is stated for twisted forms over $\delta$-Hopf-Galois extensions, which includes Picard-Vessiot ring extensions, sometimes it is easier to work with twisted forms over Picard-Vessiot field extensions.

\begin{theorem}\label{thm:tf-over-fields}
Let $R/F$ be a Picard-Vessiot ring extension with $\delta$-Galois group $\Gamma$. Let $K=\Frac(R)$, $\pmb{\Phi}$ be a type, and $M$ be a $\pmb{\Phi}$-object over $F$. Suppose that the automorphism group of $M\otimes_{F} R$ is represented by a linear algebraic group $G$ over $C$. Then \Cref{construction:F} adapts directly to give a bijection
\begin{equation}\label{eq:tf-field}
    \TF(K/F, M)\to H^{1}_{\delta}(\Gamma, G).
\end{equation} 
\end{theorem}
\begin{proof}
We have a commutative diagram
\[
\begin{tikzcd}
\TF(R/F,M) \ar[d,"\cong"']\ar[r,"i"]& \TF(K/F,M)\ar[d]\\
H^{1}_{\delta}(\Dgal(R/F),M)\ar[r,"\cong"'] & H^{1}_{\delta}(\Dgal(K/F),M).
\end{tikzcd}
\]
where $i$ is the natural inclusion viewing an $(R/F)$-twisted form of $M$ as a $(K/F)$-twisted form of $M$.
The left vertical map is bijective by \Cref{thm:F-G-inverse}.
The bottom map is bijective as $\Dgal(R/F)=\Dgal(K/F)$.
By \Cref{prop:pv-ff}, $K/R$ is $\delta$-faithfully flat. Thus an isomorphism $N_K\cong N'_{K}$ restricts to an isomorphism $N_R\cong N'_{R}$ for any two $\delta$-$F$-modules $N$ and $N'$. 
This makes $i$ a bijection.
Thus the right vertical map in \eqref{eq:tf-field} is also a bijection.
\end{proof}

We can classify the $(F^{\PV}/F)$-twisted forms from \Cref{ex-twisted-forms} by the following cohomology set. Any Picard-Vessiot field extension $L/F$ with a Picard-Vessiot subextension $K/F$ gives a map $\Dgal(L/F)\to \Dgal(K/F)$ which in turn induces a map
\begin{equation}\label{eq:H1-maps}
    H^{1}_{\delta}(\Dgal(K/F),G)\to H^{1}_{\delta}(\Dgal(L/F),G).
\end{equation}

The maps of the form \eqref{eq:H1-maps} forms a direct system. We call its direct limit 
$$H^{1}_{\delta}(F,G) = \varinjlim H^{1}_{\delta}(\Dgal(L/F),G)$$
the \emph{absolute cohomology set} over $F$ with values in $G$.

\begin{proposition}\label{prop:absolute-cohom}
Let $\pmb{\Phi}$ be a type and $M$ a $\pmb{\Phi}$-object over $F$. Suppose that the automorphism group of $M\otimes_{F} F^{\PV}$ is represented by a linear algebraic group $G$ over $C$. Suppose  also that the map
\begin{equation}\label{eq:tf-limit}
    \varinjlim \TF(K/F,M)\to \TF(F^{\PV}/F,M)
\end{equation}
induced by inclusion maps $\TF(K/F,M)\to \TF(F^{\PV}/F,M)$, where $K/F$ are Picard-Vessiot extensions over $F$, is a bijection. Then there is a bijection of the two sets $\TF(F^{\PV}/F, M)$ and $H^{1}_{\delta}(F, G)$. 
\end{proposition}
\begin{proof}
For any Picard-Vessiot field extension $K/F$, \Cref{thm:tf-over-fields} gives a bijection $\TF(K/F,M)\cong H^{1}_{\delta}(\Dgal(K/F),G)$. Now take direct limits over the Picard-Vessiot field extensions $K/F$.
\end{proof}

\begin{corollary}[$\delta$-modules]\label{cor:diff-mods-as-twisted-forms}
There is a bijection 
$\diffmod_{n}(F)\cong H^{1}_{\delta}(F,\GL_{n})$.
\end{corollary}
\begin{proof}
By \Cref{ex-twisted-forms}, it suffices to show $\TF(F^{\PV}/F,M)\cong H^{1}_{\delta}(F,\GL_{n})$, where $M$ is the trivial $\delta$-module over $F$ of rank $n$.
We first verify that the hypotheses of \Cref{prop:absolute-cohom} hold. Let $D$ be a $C$-algebra. By \Cref{lemma:automorphism_split}, $\Aut^{\delta, \bphi}_{F \otimes_C D} (M \otimes_C D) = \Aut_D(M \otimes_C D) = \GL_n(D)$.

Next a $(F^{\PV}/F)$-twisted form $N$ of $M$ gives an isomorphism $\varphi\colon N\otimes_{F}F^{\PV}\to  M\otimes_{F}F^{\PV}$. Pick $F$-bases $\{n_{i}\}$ and $\{m_{j}\}$ for $N$ and $M$. Then $\varphi(n_{i}) = \sum c_{ij}m_{j}$ for some $c_{ij}\in F^{\PV}$. Therefore $\varphi$ restricts to an isomorphism $N\otimes_{F}K\cong M\otimes_{F}K$, where $K/F$ is the Picard-Vessiot extension in $F^{\PV}$ generated by the entries $c_{ij}$ over $F$. Thus \eqref{eq:tf-limit} is surjective. 

Since all hypotheses are verified, \Cref{prop:absolute-cohom} now gives the desired bijection. 
\end{proof}

\begin{corollary}[$\delta$-torsors]\label{cor:diff-tors-as-twisted-forms}
Let $G$ be a linear algebraic group over $C$ and consider $G_F$ as a trivial $\delta$-$G_{F}$-torsor. Then there is a bijection 
$\delta\text{-}G\tors(F)\cong H^{1}_{\delta}(F,G)$.
\end{corollary}
\begin{proof}
By \Cref{ex:tors-twisted-forms}, it suffices to show $\TF(F^{\PV}/F, G_{F})\cong H^{1}_{\delta}(F,G)$.
Let $D$ be a $C$-algebra. By \Cref{lemma:automorphism_split}, $\Aut^{\delta, \bphi}_{F \otimes_C D}(G_{F\otimes_C D}) = \Aut_D(G_D) = G(D)$.

Next if $X$ is a $(F^{\PV}/F)$-twisted form of $G_{F}$, we get an isomorphism of $\delta$-Hopf-Galois extensions $F[X]\otimes_{F}F^{\PV}\cong F[G]\otimes_{F}F^{\PV}$. As in the proof of \Cref{cor:diff-mods-as-twisted-forms}, this isomorphism restricts to one over a Picard-Vessiot field extension $K/F$. Therefore $X$ is a $(K/F)$-twisted form of $G_{F}$. We conclude that \eqref{eq:tf-limit} is surjective, with $M = F[G]$.

Since all hypotheses are verified, \Cref{prop:absolute-cohom} now gives the desired bijection. 
\end{proof}

\begin{remark}\label{rem:torsor-cocycle}
We can describe the bijection in \Cref{cor:diff-tors-as-twisted-forms} explicitly. A $\delta$-$G$-torsor $X$ in $\TF(F^{\PV}/F, G_{F})$ lies in $\TF(K/F, G_{F})$ for some Picard-Vessiot field extension $K/F$. For any such $K$, $X_{K}\cong G_{K}$ and so has a $\delta$-$K$-point $x$. Any $\sigma\in \Dgal(K/F)$ defines an automorphism on $X(K)$, and we set $a_{\sigma}\in G(K)$ to be the unique element such that $\sigma(x) = x\cdot a_{\sigma}$ in $X(K)$. We can check that $\sigma\mapsto a_{\sigma}$ defines a cocycle $a\in H^{1}_{\delta}(\Dgal(K/F), G)$, and so we have a map 
\[
    \TF(F^{\PV}/F, G_F)\to H^{1}_{\delta}(\Dgal(K/F), G).
\]
The cocycle constructed is compatible with the maps 
\[
    H^{1}_{\delta}(\Dgal(K/F), G)\to H^{1}_{\delta}(\Dgal(L/F), G).
\]
Thus we get a map 
\[
    \TF(F^{\PV}/F, G_F)\to H^{1}_{\delta}(F,G).
\]
\end{remark}

Cohomological computations allows us to deduce some facts about Picard-Vessiot extensions.

\begin{proposition}
    Let $U$ be a unipotent group and $\Gamma$ a reductive group over $C$. If a $\delta$-$U_F$-torsor $X$ is split by a Picard-Vessiot extension $R/F$ with $\Dgal(R/F)=\Gamma$, then $X$ is a trivial $\delta$-$U_F$-torsor.
\end{proposition}
\begin{proof}
    By viewing $\mathbb{G}_a$ as a $\Gamma$-module, we may identify $H^1_{\delta}(\Gamma, \mathbb{G}_a)$ with the corresponding Hochschild cohomology set.
    By \cite[Corollary 14.15, Proposition 14.21]{milne_2017}, $U$ admits a central normal series 
    \begin{equation*}
        U = U_n \supset U_{n-1} \supset \cdots \supset U_1 \supset U_0 = \{e\},
    \end{equation*} with $U_{i}/U_{i-1} \cong \mathbb{G}_a$. The exact sequence of $\Gamma$-module
    \begin{equation*}
        1 \longrightarrow U_{i-1} \longrightarrow U_i \longrightarrow \mathbb{G}_a \longrightarrow 0
    \end{equation*} induces a long exact sequence
    \begin{equation*}
        \cdots \longrightarrow H^1(\Gamma, U_{i-1}) \longrightarrow H^1(\Gamma, U_{i})\longrightarrow H^1(\Gamma, \mathbb{G}_a).
    \end{equation*}
    By \cite[Proposition 1]{kemper2000characterization},
    we have $H^1_{\delta}(\Gamma, \mathbb{G}_a) = 0$, so an induction on $i$ shows that $H^{1}_{\delta}(\Gamma, U_{i}) = 0$ for $i=1,\ldots,n-1$. In particular, $H^{1}_{\delta}(\Gamma, U)=0$, so $X$ is trivial.
\end{proof}

This proposition in particular implies that a differential equation with unipotent $\delta$-Galois group cannot be fully solved over a Picard-Vessiot extension with a reductive $\delta$-Galois group unless it is already fully solved over the base field.
    
\section{Differential Central Simple Algebras}
\label{section:DCSA}

\cite{gupta2022splitting} previously showed that the $\delta$-splitting degree of $\delta$-quaternion algebras is at most three. This universal bound was also proven to be optimal for certain $\delta$-fields. In this section, we generalize this result to all $\delta$-central simple $F$-algebras when $C$ is algebraically closed. For the rest of the section, we assume any $\delta$-field extension $L$ of $F$ has the same field of constants as $F$.

We first recall some definitions and properties from \cite{Juan_Magid08}.
A \emph{$\delta$-central simple algebra} is a pair $(A, \delta)$, where $A$ is a central simple $F$-algebra and $\delta$ is a derivation on $A$ extending $\delta$ on $F$. By \cite[Proposition 1]{Juan_Magid08}, the derivations on the matrix algebra $\M_n(F)$ are given by $\delta_P$, where $P$ is an $n \times n$ traceless matrix and
\begin{equation*}
    \delta_P(x) = (x)' + Px - xP
\end{equation*}
for all $x \in \M_n(F)$. Here $(x)'$ denotes the entry-wise derivation on $x$.
We write $\DCSA_{n}(F)$ for the set of isomorphism classes of $\delta$-central simple algebras of degree $n$ over $F$.

The \emph{$\delta$-splitting degree} of a $\delta$-central simple algebra $(A, \delta)$ is defined as
\begin{equation*}
    \deg_{\spl}^{\delta}(A, \delta) = \min\{\trdeg_F(L)\mid (A, \delta) \otimes_F L \cong (\M_n(L), {\,}')\}.
\end{equation*}

\begin{theorem}\label{differential_pgln}
    There is a bijection 
    \begin{equation*}
        \DCSA_n(F) \longleftrightarrow \difftors{\PGL_{n,F}}(F).
    \end{equation*}
\end{theorem}

\begin{proof}
    Let $(A, \delta)$ be a $\delta$-central simple $F$-algebra of degree $n$. By \cite[Theorem 1]{Juan_Magid08}, we have $(A, \delta)\otimes_{F} L \cong (\M_n(L), {\,}')$ for some Picard-Vessiot extension $L/F$. 
    Therefore, the $\delta$-central simple algebras of degree $n$ are twisted forms of $(\M_n(F^{\PV}), {\,}')$. In other words,
    \[
        \DCSA_{n}(F) = \TF(F^{\PV}/F, (\M_n(F), {\,}')).
    \]
    Since $(\M_{n}(F^{\PV}),{\,}') = (\M_{n}(C),{\,}')\otimes_{C}F^{\PV}$, \Cref{lemma:automorphism_split} gives 
    \[
        \Aut^{\delta}_{F^{\PV}}(\M_n(F^{\PV}), {\,}') = \Aut_C(\M_n(C)) = \PGL_n(C).\]
    By \Cref{prop:absolute-cohom}, we have the bijection 
    \begin{equation*}
        \TF(F^{\PV}/F, (\M_n(F), {\,}')) \longleftrightarrow H^1_{\delta}(F, \PGL_n(C)).
    \end{equation*}
    Combined with \Cref{cor:diff-tors-as-twisted-forms}, we get the desired bijection
    \begin{equation*}
        \DCSA_n(F) \longleftrightarrow \difftors{\PGL_{n,F}}(F).
    \end{equation*}
\end{proof}

Theorems \ref{prop:split_torsor} and \ref{differential_pgln} together give us a description of the $\delta$-splitting degree for each $\delta$-central simple algebra, leading to the following universal bound for $\delta$-splitting degree for all $\delta$-central simple $F$-algebras.

\begin{corollary}\label{cor:bound_DCSA}
    Let $(A, \delta)$ be a $\delta$-central simple $F$-algebra of degree $n$. Then 
    \[
    \deg_{\spl}^{\delta}(A, \delta) \leq \dim(\PGL_n) = n^2 - 1.
    \] This bound is optimal when $\PGL_n$ appears as a $\delta$-Galois group over $F$.
\end{corollary}
\begin{proof}
    The first claim follows from \Cref{prop:split_torsor} and \Cref{differential_pgln}. If $\PGL_n$ appears as a $\delta$-Galois group, there exists a Picard-Vessiot ring $R$ such that $Y \coloneqq \Spec(R)$ is a simple $\PGL_n$-torsor, whence the second claim also follows from \Cref{prop:split_torsor}. 
\end{proof}

\begin{remark}\label{sharp}
    When $F = (\mathbb{C}(t), d/dt))$, all linear algebraic groups appear as $\delta$-Galois groups (see \cite{TT79}) and therefore the universal bound above is optimal. 
    
    When $F = (\mathbb{C}((t)), t (\frac{d}{dt}))$, the structure of the universal Picard-Vessiot ring (\cite[Chapter 3]{vdPS03}) suggests that the Picard-Vessiot extensions over $\mathbb{C}((t))$ must be Liouvillian and therefore the $\delta$-Galois groups must be solvable (\cite[Chapter 1 Section 5]{vdPS03}). By a slight modification of Kovacic's proof of \cite[Proposition 20]{kovacic} using the derivation $t(\frac{d}{dt})$, we may show that all $\delta$-Galois groups must be in the form $\mathbb{G}_m^s \times \mathbb{G}_a^{\epsilon} \times G$, where $\epsilon \in \{0, 1\}$ and $G$ is finite and cyclic. A different treatment is also available in \cite{Vel23}. Hence, the universal bound give in \Cref{cor:bound_DCSA} can be improved. To obtain a sharper bound,  by \Cref{prop:split_torsor}, it suffices to find a connected closed subgroup of $\PGL_n$ of the maximal dimension that is of the above form. 

    A commutative linear algebraic group is isomorphic to the direct product of a torus and a commutative unipotent group (\cite[Corollary 16.15]{milne_2017}). Therefore, a connected subgroup $G$ of $\PGL_n$ that appears as a $\delta$-Galois group must be in the form $\mathbb{G}_m^s$ or $\mathbb{G}_m^s \times \mathbb{G}_a$ for some $s$ . 
    Here $s$ is bounded by $n - 1$, the dimension of the maximal torus of $\PGL_n$. 
    See \cite[Theorem 22.6(ii) and section 23.2]{Borel}. 
    However, $\mathbb{G}_m^{n-1} \times \mathbb{G}_a$ is not a subgroup of $\PGL_n$; 
    otherwise, $\mathbb{G}_m^{n-1} \times \mathbb{G}_a$ lifts to $\mathbb{G}_m^{n} \times \mathbb{G}_a$ as a subgroup of $\GL_n$, but this is impossible. Therefore, the universal bound is at most $n - 1$.       
\end{remark}

\bibliographystyle{alpha}

\begin{thebibliography}{BHHW18}

\bibitem[Art99]{artin1999noncommutative}
Michael Artin.
\newblock Noncommutative rings.
\newblock {\em Class Notes, fall}, pages 75--77, 1999.

\bibitem[BHHW18]{BHHW18}
Annette Bachmayr, David Harbater, Julia Hartmann, and Michael Wibmer.
\newblock {Differential embedding problems over complex function fields}.
\newblock {\em Documenta Mathematica}, 23:241--291, 2018.

\bibitem[Bor69]{Borel}
Armand Borel.
\newblock {\em Linear algebraic groups}.
\newblock Graduate text in mathematics. Springer New York, 2 edition, 1969.

\bibitem[Dyc08]{dyckerhoff2008inverse}
Tobias Dyckerhoff.
\newblock {The inverse problem of differential Galois theory over the field
  $\mathbb R(z)$}.
\newblock {\em arXiv preprint}, 2008.
\newblock \href{https://arxiv.org/abs/0802.2897}{arXiv: 0802.2897}.

\bibitem[GKS22]{gupta2022splitting}
Parul Gupta, Yashpreet Kaur, and Anupam Singh.
\newblock Splitting of differential quaternion algebras.
\newblock {\em arXiv preprint}, 2022.
\newblock \href{https://arxiv.org/abs/2210.02103}{arXiv: 2210.02103}.

\bibitem[JM08]{Juan_Magid08}
Lourdes Juan and Andy~R. Magid.
\newblock {Differential Central Simple Algebras and Picard-Vessiot
  Representations}.
\newblock {\em Proceedings of the American Mathematical Society},
  136(6):1911--1918, 2008.
\newblock \url{http://www.jstor.org/stable/20535372}.

\bibitem[Kem00]{kemper2000characterization}
Gregor Kemper.
\newblock A characterization of linearly reductive groups by their invariants.
\newblock {\em Transformation groups}, 5:85--92, 2000.

\bibitem[Kol73]{kolchin1973differential}
Ellis~Robert Kolchin.
\newblock {\em {Differential algebra \& algebraic groups}}, volume~54.
\newblock Academic press, 1973.

\bibitem[Kov69]{kovacic}
J.~Kovacic.
\newblock The inverse problem in the galois theory of differential fields.
\newblock {\em Annals of Mathematics}, 89(3), 1969.

\bibitem[Mag02]{magid2002picard}
Andy~R Magid.
\newblock {The Picard--Vessiot closure in differential Galois theory}.
\newblock {\em Banach Center Publications}, 58:151--164, 2002.
\newblock \url{https://eudml.org/doc/281834}.

\bibitem[MP23]{MP}
David Meretzky and Anand Pillay.
\newblock{Picard-Vessiot extensions, linear differential algebraic groups and their torsors}.
\newblock{\em arXiv preprint}, 2008.
\newblock\href{https://arxiv.org/abs/2307.14948}{arXiv:2307.14948}.

\bibitem[Mil17a]{milne_2017}
J.~S. Milne.
\newblock {\em {Algebraic Groups: The Theory of Group Schemes of Finite Type
  over a Field}}.
\newblock Cambridge Studies in Advanced Mathematics. Cambridge University
  Press, 2017.

\bibitem[Mil17b]{milne2017primer}
James Milne.
\newblock A primer of commutative algebra.
\newblock \url{https://www.jmilne.org/math/xnotes/CA402.pdf}, 2017.

\bibitem[MS21]{masuoka2021twisted}
Akira Masuoka and Yuta Shimada.
\newblock {Twisted Forms of Differential Lie Algebras over $\mathbb{C}(t)$
  Associated with Complex Simple Lie Algebras}.
\newblock {\em Arnold Mathematical Journal}, 7(1):107--134, 2021.

\bibitem[Nar12]{nardin2012essential}
Denis Nardin.
\newblock {The essential dimension of finite group schemes}.
\newblock Master's thesis, 2012.
\newblock
  \url{https://homepages.uni-regensburg.de/~nad22969/dispense/tesimagistrale.pdf}.

\bibitem[NW07]{nuss2007non}
Philippe Nuss and Marc Wambst.
\newblock {Non-abelian Hopf cohomology}.
\newblock {\em Journal of Algebra}, 312(2):733--754, 2007.
\newblock \url{https://doi.org/10.1016/j.jalgebra.2006.10.005}.

\bibitem[Sch90]{schneider1990principal}
Hans-J{\"u}rgen Schneider.
\newblock {Principal homogeneous spaces for arbitrary Hopf algebras}.
\newblock {\em Israel Journal of Mathematics}, 72(1-2):167--195, 1990.
\newblock \url{https://doi.org/10.1007/BF02764619}.

\bibitem[TT79]{TT79}
Carol Tretkoff and Marvin Tretkoff.
\newblock Solution of the inverse problem of differential galois theory in the
  classical case.
\newblock {\em American Journal of Mathematics}, 101(6):1327--1332, 1979.

\bibitem[vdPS03]{vdPS03}
Marcus van~der Put and Michael~F. Singer.
\newblock {\em {Galois Theory of Linear Differential Equations}}.
\newblock Grundlehren der mathematischen Wissenschaften. Springer Berlin,
  Heidelberg, 2003.

\bibitem[VI23]{Vel23}
Santiago Velazquez~Iannuzzelli.
\newblock {Local differential Galois groups}.
\newblock Master's thesis, University of Pennsylvania, 2023.

\end{thebibliography}

\medskip

\noindent Author information:

\medskip

\noindent Man Cheung Tsui: Department of Mathematics,
Florida State University,
Tallahassee, FL 32304,
USA.\\ email: {\tt mktsui@fsu.edu}

\medskip

\noindent Yidi Wang: Department of Mathematics, University of Pennsylvania, Philadelphia, PA 19104, USA.\\ email: {\tt yidiwang@math.upenn.edu}

\end{document}